\documentclass[11pt,a4paper]{article}

\usepackage{fancyhdr}
\usepackage[english]{babel}
\usepackage{times}
\usepackage{graphicx}
\usepackage{amscd}
\usepackage{amsmath}
\usepackage{amsrefs}
\usepackage{amsfonts}
\usepackage{amssymb}
\usepackage{amsthm}
\usepackage{latexsym}
\usepackage{eufrak}

\newcommand{\pan}{\mathfrak{P}}
\newcommand{\vel}{\mathcal{V}}
\newcommand{\wgotica}{\mathfrak{w}}
\newcommand{\sgotica}{\mathfrak{s}}

\newtheorem{theorem}{Theorem}[section]

\newtheorem{example}[theorem]{Example}

\numberwithin{equation}{section}

\newcommand{\R}{\mathbb{R}}
\newcommand{\N}{\mathbb{N}}
\renewcommand{\epsilon}{\varepsilon}
\newcommand{\eps}{\varepsilon}

\renewcommand{\le}{\leqslant}

\renewcommand{\ge}{\geqslant}

\usepackage{authblk}
\begin{document}

\author[$\sharp$,1]{Serena Dipierro}

\author[2,3]{Luca Lombardini}

\author[2,4]{Pietro Miraglio}

\author[1,2,5,6]{Enrico Valdinoci}

\affil[$\sharp$]{\footnotesize Corresponding author\medskip }

\affil[1]{\footnotesize School of Mathematics and Statistics,
University of Melbourne,
Richard Berry Building,
Parkville VIC 3010,
Australia \medskip } 

\affil[2]{\footnotesize Dipartimento di Matematica, Universit\`a degli studi di Milano,
Via Saldini 50, 20133 Milan, Italy \medskip }

\affil[3]{\footnotesize Facult\'e des Sciences,
Universit\'e de Picardie Jules Verne, 
33, rue Saint-Leu,
80039 Amiens CEDEX 1, France\medskip}

\affil[4]{\footnotesize
Departament de Matem\`atica Aplicada I,
Universitat Polit\`ecnica de Catalunya,
Diagonal 647,
08028 Barcelona,
Spain\medskip}

\affil[5]{\footnotesize Weierstra{\ss} Institut f\"ur Angewandte
Analysis und Stochastik, Mohrenstra{\ss}e 39, 10117 Berlin, 
Germany \medskip}

\affil[6]{\footnotesize Istituto di Matematica Applicata e Tecnologie Informatiche,
Consiglio Nazionale delle Ricerche,
Via Ferrata 1, 27100 Pavia, Italy \medskip }

\title{The Phillip Island penguin parade \\
(a mathematical treatment)\thanks{Emails:
{\tt sdipierro@unimelb.edu.au}, {\tt luca.lombardini@unimi.it},
{\tt pietro.miraglio@unimi.it}, {\tt enrico@mat.uniroma3.it} }}

\date{}

\maketitle

\noindent{\it 2010 Subject Classification:} 92B05, 92B25, 37N25.

\bigskip\bigskip

\begin{abstract}
Penguins are flightless, so they are forced to walk while on land. 
In particular, they show rather specific behaviors in their homecoming,
which are interesting to
observe and to describe analytically.

In this paper, we present a simple mathematical formulation to
describe the little penguins parade in Phillip Island.

We observed that penguins have the tendency
to waddle back and forth on the shore to create
a sufficiently large group and then walk home compactly together.

The mathematical framework that we introduce describes
this phenomenon, by taking into account
``natural parameters'', such as the eye-sight of the penguins,
their cruising speed and the possible ``fear'' of animals. On the one hand,
this favors the formation of conglomerates of penguins that gather together,
but, on the other hand, this may lead to the ``panic'' of isolated and exposed
individuals.

The model that we propose is based on a set
of ordinary differential equations. Due to the discontinuous
behavior of the speed of the penguins, the mathematical
treatment (to get existence and uniqueness of
the solution) is based on a ``stop-and-go'' procedure.

We use this setting to provide rigorous examples in which
at least some penguins manage to safely return home (there are also cases
in which some penguins freeze due to panic).

To facilitate the intuition of the model,
we also present some simple numerical simulations
that can be
compared with the actual movement of the penguins parade.
\end{abstract}

\section{Introduction}\label{INTRO}

The goal of this paper is to provide a simple, but rigorous, mathematical
model which describes the formation of groups of penguins
on the shore at sunset. 


The results that we obtain are:
\begin{itemize}
\item The construction of a mathematical model
to describe the formation of groups of penguins on the shore
and their march towards their burrows; this model is based
on systems of ordinary differential equations, with a number of degree
of freedom which is variable in time (we show that the model
admits a unique solution, which needs to be appropriately defined).
\item Some rigorous mathematical results
which provide sufficient conditions for a group of penguins to reach the burrows.
\item Some numerical simulations which
show that the mathematical model well predicts, at least\footnote{It would be
desirable to have empirical data about the formation of penguins clusters
on the shore and their movements, in order to compare and
adapt the model to experimental
data and possibly give a quantitative description of concrete
scenarios.}
at a qualitative level,
the formation of clusters of penguins and their march towards the burrows;
these simulations are easily implemented by images and videos.
\end{itemize}
The methodology used is based on:
\begin{enumerate} 
\item direct observations on site, strict
interactions with experts in biology
and penguin ecology, 
\item mathematical formulation
of the problem and rigorous deductive arguments, and 
\item numerical simulations.
\end{enumerate}
In this introduction, we will describe the ingredients which lead
to the construction of the model, presenting its basic features and also its limitations.
Given the interdisciplinary flavor of the subject, it is
not possible to completely split the biological discussion
from the mathematical formulation, but we can mention that:
\begin{itemize}
\item The main mathematical equation is given in formula~\eqref{EQ}.
Before~\eqref{EQ}, the main ingredients coming from live observations
are presented. After~\eqref{EQ}, the mathematical quantities
involved in the equation are discussed and elucidated.
\item The existence and uniqueness theory for equation~\eqref{EQ}
is presented in Section~\ref{EXUT}.
\item Some rigorous mathematical results about equation~\eqref{EQ}
are given in Section~\ref{HOME} (roughly speaking,
these are results which give sufficient conditions
on the initial conditions of the system and on the external environment
for the successful homecoming of the penguins, and their precise
formulation requires the development of the mathematical framework
in~\eqref{EQ}).
\item In Section~\ref{VIDEO} we present numerics, images and videos which favor the intuition and
set the mathematical model of~\eqref{EQ} into
a concrete framework, which is easily comparable with the
real-world phenomenon.
\end{itemize}
Before that, we think that it is important to describe our experience of
the penguins parade in Phillip Island, both to allow the
reader which is not familiar with the event to concretely take part in it,
and to describe some peculiar environmental aspects which are
crucial to understand our description (for instance,
the weather in Phillip Island is completely different from the Antarctic one,
so many of our considerations are meant to be limited to this particular 
habitat) -- also, our personal experience
in this bio-mathematical adventure is a crucial point, in our opinion,
to describe how
scientific curiosity can trigger academic
activities. The reader who is not interested in the description of
the penguins parade in Phillip Island can skip this part and go directly 
to Subsection~\ref{QUIQUI}.

\subsection{Description of the penguins parade}
An  extraordinary event in the state of Victoria, Australia,
consists in the march of the little penguins
(whose scientific name is ``Eudyptula minor'') who live
in Phillip Island. At sunset, when it gets too dark for the little penguins to
hunt their food in the sea, they come out
to return to their homes
(which are small cavities in the terrain, that are located
at some dozens of meters from the water edge).

As foreigners in Australia, our first touristic trip
in the neighborhoods of Melbourne consisted in a
one-day excursion to Phillip Island, enjoying
the presence of wallabies, koalas and kangaroos, visiting
some farms during the trip, walking on the spectacular
empty beaches of the coast and -- cherry on top --
being delighted by the show
of the little penguins parade.

Though at that moment
we were astonished by the poetry of
the natural exhibition of the penguins, later on, driving back to Melbourne
in the middle of the night,
we started thinking back to what we saw and
attempted to understand the parade from a rational,
and not only emotional, point of view
(yet we believe that the rational approach was not diminishing
but rather enhancing the sense of our intense experience).
What follows is indeed the mathematical description that came out
of the observations on site at Phillip Island, enriched by the scientific discussions
we later had with 
penguin ecologists.

\subsection{Observed behaviors of penguins in the parade}\label{QUIQUI}

By watching the penguins parade in
Phillip Island, it
seemed to us that some simple
features appeared in the very unusual pattern followed by the little penguins:
\begin{itemize}
\item Little penguins have the strong tendency to
gather together in a sufficiently large number before starting their
march home.
\item They have the tendency to march on a straight line,
compactly arranged
in a cluster, or group.
\item To make this group, they will move back
and forth, waiting for other fellows or even going back to
the sea if no other mate is around.
\item If, by chance or by mistake, a little penguin remains isolated,
(s)he can panic\footnote{In this paper, the use of
terminology such as ``panic'' has to be intended in a strictly mathematical sense:
namely, in the equation that we propose, there is a term which makes
the velocity stop. The use of the word ``panic'' is due to the fact that
this interruption in the penguin's movement is not due to physical impediments, but rather to the fact that no other
penguin is in a sufficiently small neighborhood. Notice that it may be the case that
a ``penguin in panic'' is not really in danger; simply, at a mathematical level,
a quantified version of the notion of ``isolation'' leads the penguin to stop.
For this reason, we think that the word ``panic''
is rather suggestive
and actually sufficiently appropriate to describe a psychological 
attitude of the animal which could also turn out to be not
necessarily convenient from the purely rational point of view:
indeed, a ``penguin in panic'' risks to remain more exposed than
a penguin that keeps moving (hence panic seems to be a good word to
describe psychological uneasiness, with or without direct
cause, that produces hysterical or irrational behavior).

It would be desirable to have further non-invasive tests to measure
how the situation that we describe by the word ``panic'' is felt by the penguin
at an emotional level (at the moment, we are not aware of experiments like
this in the literature). Also, it would be highly desirable to have some
precise experiments to determine how many penguins do not manage
to return to their burrows within
a certain time
after dusk and stay either in the water or in the vicinity of the shore.

On the one hand, in our opinion, it is likely that rigorous experiments on site will
demonstrate that the phenomenon for which an isolated penguin stops
is rather uncommon in nature (and, when it happens,
it may be unrelated to emotional feelings such as fear). On the other hand, our model is general
enough to take into account
the possibility that a penguin stops its march, and, at a quantitative level,
we emphasized this ``panic'' feature in the pictures of
Section \ref{VIDEO} to make the situation visible.

The reader who does not want to take into account the panic function in the model
can just set this function to be identically equal to~$1$
(the mathematical formulation of this remark will be given
in footnote~\ref{NOPANIC}). In this particular case,
our model will still exhibit the formation of groups of penguins moving together.}
and this fact
can lead to a complete\footnote{Though no experimental test has been run on this
phenomena,
in the parade that we have seen live it indeed happened
that one little penguin remained isolated from the others
and ``panic prevailed'': even though (s)he was absolutely fit
and no concrete obstacle
was obstructing the motion, (s)he got completely stuck for
half an hour and the staff of the Nature Park had
to go and provide assistance. We stress again that
the fact that the penguin stopped moving did not seem to
be caused by any physical impediment
(as confirmed to us by the Ranger on site), since no extreme environmental condition
was occurring,
the animal
was not underweight, and was able to come out of the water
and move effortlessly on the shore autonomously
for about~$15$ meters, before suddenly stopping.}
freeze.
\end{itemize}
For a short video (courtesy of Phillip Island Nature Parks)
of the little penguins parade, in which the formation
of groups is rather evident, see e.g. \\
{\tt 
https://www.ma.utexas.edu/users/enrico/penguins/Penguins1.MOV}
\medskip

The simple features listed above are likely to be a consequence of
the morphological structure of the little penguins and of
the natural environment. As a matter of fact,
little penguins are a marine-terrestrial species.
They are highly efficient swimmers
but possess a rather inefficient form of locomotion on land
(indeed, flightless penguins, as the ones
in Phillip Island, 
waddle, more than walk). At dusk,
about 80 minutes after sunset according to the data in~\cite{CH5},
little penguins return ashore
after their fishing activity in the sea.
Since their bipedal locomotion is slow and rather goofy
(at least from the human subjective perception, but also in comparison with
the velocity or agility that is well known to be
typical of predators in nature), and the
easily
recognizable countershading of the penguins
is likely to make them visible to predators,
the transition
between the marine and terrestrial environment
may be particularly stressful\footnote{At the moment,
there seems to be no complete experimental evidence measuring the ``danger
perceived
by the penguins'', and, of course, the words ``danger'' 
and ``fear'' have to be intended -- for the purposes of this paper, and strictly speaking --
as a human
interpretation, without experimental testing, and thus between ``quotation marks''. Nevertheless,
given the
swimming
ability of the penguins and the environmental conditions,
one may well conjecture that
an area of high danger for a penguin is the one adjacent to the
shore-line,
since this is a habitat which provides little or no shelter,
and it is also in a regime of reduced visibility.

As a matter of fact, we have been told that
the Rangers in Phillip Island implemented a control on the
presence of the foxes in the proximity of the shore, with the aim of
limiting the number of possible predators.

Whether the penguins really feel an emotion comparable to what humans call ``fear''
or ``panic'' is not within the goals of this paper (it is of course also possible that,
at a neurological level, the behavior of the penguins follows
different patterns than human emotions). Nevertheless we use here
the words ``fear'' and ``danger'' to give an easy-to-communicate
justification of the mathematical model. Of course, any progress in the
study of the emotional behavior of penguins would be highly desirable in this sense.}
for the penguins 
(see~\cite{CH6}) and this fact is probably related to the
formation of penguins groups 
(see e.g.~\cite{CH1}). 
Thus, in our opinion, the rules
that we have listed may be seen as the outcome
of the difficulty of the little penguins to perform their transition
from a more favorable environment to an habitat in which
their morphology turns out to be suboptimal.

\subsection{Mathematical formulation}

To translate  into a mathematical framework the
simple observations on the penguins behavior that we
listed in Subsection~\ref{QUIQUI},
we propose the following equation:
\begin{equation} \label{EQ}
\dot{p}_i(t) = \pan_i\big(p(t),w(t);t\big) \,\Big( \eps+\vel_i\big(
p(t),w(t);t\big)
\Big)+f\big(p_i(t),t\big).\end{equation}
Here, the following notation is used:
\begin{itemize}
\item The function~$n:[0,+\infty)\to\N_0$, where $\N_0:=\N\setminus\{0\}$,
is piecewise constant and nonincreasing, namely
there exist a (possibly finite)
sequence~$0=t_0<t_1<\dots<t_j<\dots$ 
and integers~$n_1>\dots>n_j>\dots$
such that~$n(t)=n_j\in\N_0$ for any~$t\in(t_{j-1},t_j)$.
\item
At time~$t\ge0$, there is a set of~$n(t)$
groups of penguins~$p(t)=\big(p_1(t),\dots,p_{n(t)}(t)\big)$.
That is, at time~$t\in(t_{j-1},t_j)$ there is a set
of~$n_j$ 
clusters of penguins~$p(t)=\big(p_1(t),\dots,p_{n_j}(t)\big)$.
\item For any~$i\in\{1,\dots,n(t)\}$, the coordinate~$p_i(t)\in\R$
represents the position of a group of penguins
on the real line:
each of these groups contains a certain number
of little penguins, and this number is denoted by~$w_i(t)\in\N_0$.
We also consider the array~$w(t)=
\big(w_1(t),\dots,w_{n(t)}(t)\big)$.

We assume that~$w_i$ is piecewise constant, namely that~$w_i(t)=\bar w_{i,j}$
for any~$t\in(t_{j-1},t_j)$, for some~$\bar w_{i,j}\in\N_0$,
namely the number of little penguins in each group remains
constant, till the next penguins join the group at time~$t_j$ 
(if, for the sake of simplicity,
one wishes to think that initially all the little penguins are
separated one from the other, one may also suppose that~$w_i(t)=1$
for all~$i\in\{1,\dots,n_1\}$ and~$t\in[0,t_1)$).

Up to renaming the variables, we suppose that the initial position of the groups
is increasing with respect to the index, namely
\begin{equation}\label{INI P}
p_1(0)<\dots<p_{n_1}(0).
\end{equation}
\item The parameter~$\eps\ge0$ represents
a drift velocity of the penguins towards their house,
which is located\footnote{For concreteness, if~$p_i(T)=H$ for some~$T\ge0$,
we can set~$p_i(t):=H$ for all~$t\ge T$ and remove~$p_i$ \label{STAY}
from the equation of motion -- that is, the penguin has safely come back
home and (s)he can go to sleep. In real life penguins have some social life before going to sleep, but we are not taking this under consideration for the moment.}
at the point~$H\in(0,+\infty)$.
\item For any~$i\in\{1,\dots,n(t)\}$,
the quantity~$\vel_i\big(p(t),w(t);t\big)$
represents the strategic velocity of the~$i$th
group of penguins and 
it can be considered as a function with domain varying in time
$$ \vel_i(\cdot,\cdot;t):\R^{n(t)}\times\N^{n(t)}\to \R,$$
i.e.
$$ \vel_i(\cdot,\cdot;t):\R^{n_j}\times\N^{n_j}\to \R
\quad{\mbox{ for any }}t\in(t_{j-1},t_j),$$
and,
for any~$(\rho,w)=
(\rho_1,\dots,\rho_{n(t)},w_1,\dots,w_{n(t)})
\in\R^{n(t)}\times\N^{n(t)}$,
it is of the form
\begin{equation}\label{VELOCITY}
\vel_i\big(\rho,w;t\big):=\Big(1-\mu\big(w_i\big)\Big)
\,m_i\big(\rho,w;t\big)
+v\mu\big(w_i\big).\end{equation}
In this setting, for any~$(\rho,w)=
(\rho_1,\dots,\rho_{n(t)},w_1,\dots,w_{n(t)})
\in\R^{n(t)}\times\N^{n(t)}$,
we have that
\begin{equation}\label{emme i} 
m_i\big(\rho,w;t\big):=
\sum_{j\in\{1,\dots,n(t)\}} {\rm sign}\,(\rho_j-\rho_i)\;w_j\;
\sgotica(|\rho_i-\rho_j|),
\end{equation}
where~$\sgotica\in {\rm Lip}([0,+\infty))$ is
nonnegative and nonincreasing and, as usual,
we denoted the ``sign function'' as
$$ \R\ni r\mapsto {\rm sign}\,(r):=\left\{
\begin{matrix}
1 & {\mbox{ if }} r>0,\\
0 & {\mbox{ if }} r=0,\\
-1& {\mbox{ if }} r<0.
\end{matrix}
\right.$$
Also, for any~$\ell\in\N$, we set
\begin{equation}\label{MU} \mu(\ell):=\left\{
\begin{matrix}
1 & {\mbox{ if }} \ell\ge\kappa,\\
0 & {\mbox{ if }} \ell\le\kappa-1,
\end{matrix}
\right.
\end{equation}
for a fixed~$\kappa \in\N$, with~$\kappa\ge2$,
and~$v>\eps$.

In our framework, the meaning of the
strategic velocity of the~$i$th
group of penguins is the following:
\begin{itemize}
\item When the group of penguins is too small
(i.e. it contains less than~$\kappa$ little penguins),
then the term involving~$\mu$ vanishes, thus the
strategic velocity reduces to the term given by~$m_i$;
this term, in turn, takes into account the position of the other
groups of penguins.
That is, each penguin is endowed with a ``eye-sight''
(i.e., the capacity of seeing the other penguins that are ``sufficiently close''
to them), which is
modeled by the function~$\sgotica$
(for instance, if~$\sgotica$ is identically equal to~$1$, then
the penguin has a ``perfect eye-sight'';
if~$\sgotica(r)=e^{-r^2}$, then the penguin sees close
objects much better than distant ones; if~$\sgotica$
is compactly supported, then the penguin does not see too far
objects, etc.). Based on the position of the other mates that (s)he
sees, the penguin has the tendency to move either forward
or backward (the more penguins (s)he sees ahead, the more
(s)he is inclined to move forward, 
the more penguins (s)he sees behind, the more
(s)he is inclined to move backward, and nearby penguins
weight more than distant ones, due to the monotonicity
of~$\sgotica$). This strategic tension coming from the position
of the other penguins is encoded by the function~$m_i$.
\item When the group of penguins is sufficiently large
(i.e. it contains at least~$\kappa$ little penguins),
then the term involving~$\mu$ is equal to~$1$;
in this case,
the
strategic velocity is~$v$
(that is, when the group of penguins is 
sufficiently rich in population, its strategy
is to move forward with cruising speed equal to~$v$). 
\end{itemize}
\item The function~$\pan_i\big(p(t),w(t);t\big)$
represents the panic that the $i$th group
of penguins fears in case of extreme isolation 
from the rest of the herd. Here,
we take~$\overline d>\underline d>0$,
a nonincreasing\footnote{Here the notation ``${\rm Lip}$''
stands for bounded and Lipschitz continuous functions.}
function~$\varphi\in {\rm Lip}(\R,[0,1])$,
with~$\varphi(r)=1$ if~$r\le \underline d$
and~$\varphi(r)=0$ if~$r\ge \overline d$,
and, for any~$\ell\in\N_0$,
\begin{equation}\label{OMEGA} \wgotica(\ell) :=
\left\{
\begin{matrix}
1 & {\mbox{ if }} \ell\ge2,\\
0 & {\mbox{ if }} \ell=1,
\end{matrix}
\right.\end{equation}
and we take as panic function\footnote{The case of~$\varphi$
identically equal to~$1$ can be also comprised \label{NOPANIC}
in our setting. In this case, also $\pan_i$ is identically one
(which corresponds to the case in which penguins
do not panic).}
the function with variable domain
$$ \pan_i(\cdot,\cdot;t):\R^{n(t)}\times\N^{n(t)}\to[0,1],$$
i.e.
$$ \pan_i(\cdot,\cdot;t):\R^{n_j}\times\N^{n_j}\to[0,1]
\quad{\mbox{ for any }}t\in(t_{j-1},t_j),$$
given,
for any~$(\rho,w)=
(\rho_1,\dots,\rho_{n(t)},w_1,\dots,w_{n(t)})
\in\R^{n(t)}\times\N^{n(t)}$, by
\begin{equation} \label{PAN}
\pan_i\big(\rho,w;t\big) :=
\max\Big\{ \wgotica(w_i),\;
\max_{{j\in\{1,\dots,n(t)\}}\atop{j\ne i}}
\varphi\big(|\rho_i-\rho_j|\big)
\Big\}.\end{equation}
The panic function describes the fact that,
if the group gets scared, then it has the tendency
to suddenly
stop. This happens
when the group contains only one element (i.e., $\wgotica_i=0$)
and the other groups are far apart (at distance
larger than~$\overline d$).

Conversely, if the group contains at least two little penguins,
or if there is at least another group sufficiently close
(say at distance smaller than~$\underline d$), then
the group is self-confident,
namely the function~$\pan_i\big(p(t),w(t);t\big)$
is equal to~$1$ and the total 
intentional velocity of the group coincides
with the strategic velocity.

Interestingly, the panic function~$\pan_i$ may
be independent of the eye-sight function~$\sgotica$: namely
a little penguin can panic if (s)he 
feels alone and too much exposed, even if (s)he can see
other little penguins (for instance, if~$\sgotica$
is identically equal to~$1$, the little penguin always sees the other
members of the herd, still (s)he can panic if they are too far apart).
\item The function~$f\in{\rm Lip}(\R\times[0,+\infty))$
takes into account the environment. For a neutral environment,
one has that this term vanishes. In practice,
it may take into account the ebb and flow
of the sea on the foreshore (where the little penguins parade
starts), the possible
ruggedness of the terrain, the
presence of predators, etc. (as a variation, one
can consider also a stochastic version of this term).
\end{itemize}
Given the interpretation above, equation~\eqref{EQ}
tries to comprise the pattern that we described in words and
to set the scheme of motion of the little penguins
into a mathematical framework.

\subsection{Comparison with the existing literature}
We observe that, to the best of our knowledge,
there is still no specific
mathematical attempt to describe in a concise
way the penguins parade. The mathematical
literature of penguins has mostly focused on the
description of the heat flow in the penguins feathers (see~\cite{MR2899094}),
on the numerical analysis
to mark animals for later identification
(see~\cite{MR2926716}),
on the statistics of the 
Magellanic penguins at sea
(see~\cite{MR2718048}),
on the hunting strategies of fishing penguins (see~\cite{fish}),
and
on the isoperimetric arrangement of the 
Antarctic penguins to prevent the heat dispersion
caused by the polar wind and on the crystal
structures and solitary waves produced by such arrangements
(see~\cite{1367-2630-15-12-125022}
and~\cite{MR3235841}). We remark that
the climatic situation in Phillip Island
is rather different from the Antarctic one and, given
the very mild temperatures of the area, we do not think that
heat considerations should affect too much the behavior and
the moving
strategies of the Victorian little penguins and their tendency
to cluster seems more likely to be a defensive
strategy against possible predators.
\medskip

Though no mathematical formulation of the little penguins parade
has been given till now, a series
of experimental analysis has been recently performed
on the specific environment of Phillip Island.
We recall, in particular, \cite{CH1}, which describes
the association of the little penguins in groups, by collecting data
spanning over several years,
\cite{CH2}, which describes the effect of fog on the orientation
of the little penguins
(which may actually do not come back home
in conditions of poor visibility), \cite{CH3} and~\cite{CH4},
in which a data analysis is performed to show the fractal structure
in space and time for the foraging of the little penguins, also
in relation to L\'evy flights and fractional Brownian motions.

For an exhaustive list of publications focused on the behavior
of the little penguins of Phillip Island, we refer to the web page\\
{\tt https://www.penguins.org.au/conservation/research/publications/}\\
This pages contains more than~$160$ publications related
to the environment of Phillip Island, with special emphasis on the biology of little penguins.

\subsection{Discussion on the model proposed: simplifications,
generalizations and further directions of investigation}

We stress that the model proposed in~\eqref{EQ} is of course a
dramatic
simplification of ``reality''.
As often happens in science indeed, several simplifications
have been adopted in order
to allow a rigorous mathematical treatment
and handy numerical computations: nevertheless the model is
already rich enough to detect some specific features
of the little penguins parade, such as the formation of groups, the oscillatory waddling of the penguins and the possibility that panic interferes with rationally more convenient motions. Moreover, our model is flexible enough to allow
specific distinctions between the single penguins (for instance,
with minor modifications\footnote{In particular, one can replace the quantities $v,\sgotica,\mu,\kappa,\varphi$ with $v_i,\sgotica_i,\mu_i,\kappa_i,\varphi_i$ if one wants to customize these features for every group.}, one can take into account the possibility
that different penguins have a different eye-sight, that they have a different reaction
to isolation and panic, or that they exhibit some specific social behavior
that favors the formation of clusters selected by specific characteristics); 
similarly, the modeling of the habitat may also encode different possibilities
(such as the burrows of the penguins being located in different places),
and multi-dimensional models can be also constructed using similar ideas.\medskip

Furthermore, natural modifications lead to the possibility that
one or a few penguins may leave an already formed group\footnote{For
instance, rather than forming one single group, 
the model can still consider the penguins of the
cluster as separate elements, each one with its
own peculiar behavior.} (at the moment, for simplicity,
we considered here the basic model in which, 
once a cluster is made up,
it keeps moving without losing any of its elements -- we plan to
address in a future project
in detail the case of groups which may also
decrease the number of components, possibly in dependence
of random fluctuations or social considerations among
the members of the group).\medskip

In addition, for simplicity,
in this paper
we modeled each group to be located at a precise point:
though this is not a
completely unrealistic assumption (given that the scale of
the penguin is much smaller than that of the beach), one can also easily modify
this feature by locating a cluster in a region comparable to its size.\medskip

In future projects, we plan to introduce
other more sophisticated models, 
also taking into account
stochastic oscillations and optimization methods, 
and, on the long range,
to use these models
in a detailed experimental confrontation taking advantage of
the automated monitoring systems under development in Phillip Island.\medskip

The model that we propose here is also flexible enough
to allow quantitative modifications of all the parameters involved.
This is quite important, since these parameters may vary due
to different conditions of the environment. For instance,
the eye-sight of the penguins can be reduced by the fog (see~\cite{CH2}),
and by the effect of moonlight and artificial light (see~\cite{CH5}).

Similarly, the number of penguins in each group
and the velocity of the herd may vary due to 
structural changes of the beach:
roughly speaking,
from the empirical data, penguins typically gather into groups of~$ 5$--$10$
individuals (but we have also observed
much larger groups forming on the beach)
within 40 second intervals, see~\cite{CH1},
but the way these groups are built varies year by year
and, for instance, the number of individuals which
always gather into the same group changes year by year
in strong dependence with the breeding success
of the season, see again~\cite{CH1}. Also,
tidal phenomena may change the number of little
penguins in each group and the velocity of the group,
since the change of the beach width alters the perception
of the risk of the penguins. For instance, a low tide produces a larger
beach, with higher potential risk of predators, thus making
the penguins gather in groups of larger size, see~\cite{CH6}.
\medskip

{F}rom the mathematical viewpoint, we remark that~\eqref{EQ}
does not follow into the classical framework of ordinary differential
equations, since the right hand side of the equation is not
Lipschitz continuous (and, in fact, it is not even continuous).
This mathematical complication is indeed
the counterpart of the real motion of the little penguins 
in the parade, which have the tendency to change their speed
rather abruptly to maintain contact with the other elements of 
the herd. That is, on view, it does not
seem unreasonable to model, as a simplification, 
the speed of the penguin as
a discontinuous function, to take into account the sudden
modifications of the waddling
according to the position of the other penguins,
with the conclusive aim of gathering together a sufficient
number of penguins in a group which eventually will
march concurrently in the direction of their 
burrows.

\subsection{Detailed organization of the paper}

The mathematical treatment of equation~\eqref{EQ} that
we provide in this paper is the following.

\begin{itemize}
\item In Section~\ref{EXUT},
we provide a notion of solution for which~\eqref{EQ} is uniquely
solvable in the appropriate setting. This notion of solution will
be obtained by a ``stop-and-go'' procedure, which
is compatible with the idea that when two (or more)
groups of penguins meet, they form a new, bigger group
which will move coherently in the sequel of the march.
\item In Section~\ref{HOME}, we discuss a couple of
concrete examples in which the penguins are able to safely return home: namely, 
we show that there are
``nice'' conditions in which the strategy of the penguins
allows a successful homecoming.
\item In Section~\ref{VIDEO}, we present a series
of numerical simulations
to compare our mathematical model with the real-world experience.
\end{itemize}

\section{Existence and uniqueness theory for equation~\eqref{EQ}}\label{EXUT}

We stress that equation~\eqref{EQ}
does not lie within the setting of ordinary differential
equations, since the right hand side is not Lipschitz
continuous (due to the discontinuity of the functions~$w$
and~$m_i$, and in fact the right hand side also involves
functions with domain varying in time).
As far as we know, the weak formulations of 
ordinary differential
equations as the ones of~\cite{MR1022305} do not take
into consideration the setting of equation~\eqref{EQ},
so we briefly discuss here a direct
approach to the existence and uniqueness theory for
such equation. To this end, and to clarify our
direct approach, we present two illustrative examples (see e.g.~\cite{MR1028776}).

\begin{example}{\rm Setting~$x:[0,+\infty)\to\R$, the ordinary differential
equation
\begin{equation}\label{L1} \dot x(t) =\left\{
\begin{matrix}
-1 & {\mbox{ if }} x(t)\ge0,\\
1 & {\mbox{ if }} x(t)<0
\end{matrix}
\right. \end{equation}
is not well posed. Indeed, taking an initial datum~$x(0)<0$,
it will evolve with the formula~$x(t)=t+x(0)$ for any~$t\in[0,-x(0)]$
till it hits the zero value. At that point, equation~\eqref{L1}
would prescribe a negative velocity, which becomes contradictory
with the positive velocity prescribed to the negative coordinates.
}\end{example}

\begin{example}\label{EXMPL2}{\rm The ordinary differential
equation
\begin{equation}\label{L2} \dot x(t) =\left\{
\begin{matrix}
-1 & {\mbox{ if }} x(t)>0,\\
0 & {\mbox{ if }} x(t)=0,\\
1 & {\mbox{ if }} x(t)<0
\end{matrix}
\right. \end{equation}
is similar to the one in~\eqref{L1},
in the sense that it does not fit into the standard
theory of ordinary differential equations, due to the lack of
continuity of the right hand side. But, differently from
the one in~\eqref{L1}, it can be set into an existence and
uniqueness theory by a simple ``reset'' algorithm.

Namely, taking an initial datum~$x(0)<0$,
the solution
evolves with the formula~$x(t)=t+x(0)$ for any~$t\in[0,-x(0)]$
till it hits the zero value. At that point, equation~\eqref{L2}
would prescribe a zero velocity, thus a natural way to continue the
solution is to take~$x(t)=0$ for any~$t\in[-x(0),+\infty)$
(similarly,
in the case of positive 
initial datum~$x(0)>0$, a natural way to continue the solution
is~$x(t)=-t+x(0)$ for any~$t\in[0,x(0)]$
and~$x(t)=0$ for any~$t\in[x(0),+\infty)$).
The basic idea for this continuation method
is to flow the equation according
to the standard Cauchy theory of
ordinary differential equations
for as long as possible, and then, when the classical theory
breaks, ``reset'' the equation with respect of the datum
at the break time (this method is not universal
and indeed it does not work for~\eqref{L1},
but it produces a natural global solution for~\eqref{L2}).
}\end{example}

In the light of Example~\ref{EXMPL2},
we now present a framework in which equation~\eqref{EQ}
possesses a unique solution (in a suitable ``reset'' setting).
To this aim, we first notice that 
the initial number of groups of penguins is fixed to be equal to~$n_1$
and each group is given by a fixed number of little penguins packed
together (that is,
the number of little penguins in the $i$th initial group being equal to~$\bar w_{i,1}$
and~$i$ ranges from~$1$ to~$n_1$). So, we set~$\bar w_1:=(\bar w_{1,1},\dots,\bar w_{n_1,1})$
and~$\bar\wgotica_{i,1}=\wgotica(\bar w_{i,1})$, where~$\wgotica$ was defined in~\eqref{OMEGA}.
For any~$\rho=(\rho_1,\dots,\rho_{n_1})\in\R^{n_1}$, let also
\begin{equation}\label{PANNUOVA}  \pan_{i,1}(\rho)
:=
\max\Big\{ \bar\wgotica_{i,1},\;
\max_{ {j\in\{1,\dots,n_1\}}\atop{j\ne i}}
\varphi\big(|\rho_i-\rho_j|\big)
\Big\}. \end{equation} 
The reader may compare this definition with the one in~\eqref{PAN}.
For any~$i\in\{1,\dots,n_1\}$ we also
set
$$\bar\mu_{i,1}:=\mu(\bar w_{i,1}),$$ where~$\mu$ is the function defined in~\eqref{MU},
and, for any~$\rho=(\rho_1,\dots,\rho_{n_1})\in\R^{n_1}$,
\begin{equation*} 
\bar m_{i,1}(\rho):=
\sum_{j\in\{1,\dots,n_1\}} {\rm sign}\,(\rho_j-\rho_i)\;\bar w_{j,1}\;
\sgotica(|\rho_i-\rho_j|)
.\end{equation*}
This definition has to be compared with~\eqref{emme i}. Recalling~\eqref{INI P}
we also set
$$ {\mathcal{D}}_1:=\{\rho=(\rho_1,\dots,\rho_{n_1})\in\R^{n_1} {\mbox{ s.t. }}
\rho_1<\dots<\rho_{n_1}\}.$$
We remark that if~$\rho\in{\mathcal{D}}_1$ then
\begin{equation*} 
\bar m_{i,1}(\rho)=
\sum_{j\in\{i+1,\dots,n_1\}} \bar w_{j,1}\;
\sgotica(|\rho_i-\rho_j|)
-\sum_{j\in\{1,\dots,i-1\}} \bar w_{j,1}\;
\sgotica(|\rho_i-\rho_j|)
\end{equation*}
and therefore
\begin{equation}\label{VELOCITY2}
{\mbox{$\bar m_{i,1}(\rho)$ is bounded
and Lipschitz for any~$\rho\in{\mathcal{D}}_1$.}}\end{equation}
Then, we set
\begin{equation*}
\vel_{i,1}(\rho):=(1-\bar\mu_{i,1})
\,\bar m_{i,1}(\rho)
+v\bar\mu_{i,1}.\end{equation*}
This definition has to be compared with the one in~\eqref{VELOCITY}.
Notice that, in view of~\eqref{VELOCITY2}, we have that
\begin{equation}\label{VELOCITY3}
{\mbox{$\vel_{i,1}(\rho)$ is bounded and Lipschitz
for any~$\rho\in{\mathcal{D}}_1$.}}\end{equation}
So, we set
$$ G_{i,1}(\rho,t):=\pan_{i,1}(\rho)\,\big(\eps+\vel_{i,1}(\rho)\big)+
f(\rho_i,t).$$
{F}rom~\eqref{PANNUOVA} and~\eqref{VELOCITY3}, we have that~$G_{i,1}$
is bounded and Lipschitz in~${\mathcal{D}}_1\times[0,+\infty)$. Consequently, from
the global existence and uniqueness of solutions of ordinary differential equations,
we have that there exist~$t_1\in(0,+\infty]$ and a solution~$p^{(1)}(t)=(p^{(1)}_1(t),\dots,p^{(1)}_{n_1}(t))\in{\mathcal{D}}_1$
of the Cauchy problem
\begin{align*}
\begin{cases}
\,\dot p^{(1)}_i (t)= G_{i,1}\big( p^{(1)}(t),\,t\big) \qquad {\mbox{ for }}t\in(0,t_1), \\
\,p^{(1)}(0) \qquad {\mbox{ given in }}{\mathcal{D}}_1 \\
\end{cases}
\end{align*}
and
\begin{equation}\label{STOP}
p^{(1)}(t_1)\in\partial {\mathcal{D}}_1,\end{equation} see e.g. 
Theorem 1.4.1 in~\cite{MR3244289}.

The solution of~\eqref{EQ} will be taken to be~$p^{(1)}$ in~$[0,t_1)$,
that is, we set~$p(t):=p^{(1)}(t)$ for any~$t\in[0,t_1)$.
We also set that~$n(t):=n_1$ and~$w(t):=(\bar w_{1,1},\dots,\bar w_{n_1,1})$.
With this setting, we have that~$p$ is a solution of equation~\eqref{EQ}
in the time range~$t\in(0,t_1)$
with prescribed initial datum~$p(0)$. Condition~\eqref{STOP} allows us to perform
our ``stop-and-go'' reset procedure as follows: we denote by~$n_2$
the number of distinct points in the set~$\{p^{(1)}_1(t_1),\dots,p^{(1)}_{n_1}(t_1) \}$.
Notice that~\eqref{STOP} says that if~$t_1$ is finite then~$n_2\le n_1-1$
(namely, at least two penguins have reached the same position). In this way,
the set of points~$\{p^{(1)}_1(t_1),\dots,p^{(1)}_{n_1}(t_1) \}$
can be identified by the set of~$n_2$ distinct points, that we denote by~$\{p^{(2)}_1(t_1),\dots,p^{(2)}_{n_2}(t_1) \}$, with the convention that
$$ p^{(2)}_1(t_1)<\dots<p^{(2)}_{n_2}(t_1).$$
For any~$i\in\{1,\dots,n_2\}$, we also set
$$ \bar w_{i,2} := \sum_{{j\in\{1,\dots,n_1\}}\atop{
p^{(1)}_j(t_1)=p^{(2)}_i(t_1)
}} \bar w_{j,1}.$$
This says that the new group of penguins indexed by~$i$
contains all the penguins that have reached that position at time~$t_1$.

Thus, having the ``new number of groups'', that is~$n_2$,
the ``new number of little penguins in each group'', that is~$\bar w_2=(\bar w_{1,2},\dots,\bar w_{n_2,2})$,
and the ``new initial datum'', that is~$p^{(2)}(t_1)=\big(
p^{(2)}_1(t_1),\dots,p^{(2)}_{n_2}(t_1)\big)$, we can solve a new differential equation
with these new parameters, exactly in the same way as before, and keep iterating this process.

Indeed, recursively, we suppose that we have 
found~$t_1<t_2<\dots<t_k$,
$p^{(1)}:[0,t_1]\to\R^{n_1}$, $\dots$, $p^{(k)}:[0,t_k]\to\R^{n_k}$
and~$\bar w_1\in\N_0^{n_1}$, $\dots$,
$\bar w_k\in\N_0^{n_k}$ such that, setting
$$ p(t):=p^{(j)}(t)\in
{\mathcal{D}}_j, \qquad 
n(t):=n_j\qquad{\mbox{and}}\qquad
w(t):=\bar w_{j}\qquad{\mbox{for $t\in [t_{j-1},t_j)$ and~$j\in\{1,\dots,k\}$,}} $$
one has that~$p$ solves~\eqref{EQ} in each interval~$(t_{j-1},t_j)$
for~$j\in\{1,\dots,k\}$, with the ``stop condition''
$$ p^{(j)}(t_j)\in\partial {\mathcal{D}}_j,$$
where
$$ {\mathcal{D}}_j:=\{\rho=(\rho_1,\dots,\rho_{n_j})\in\R^{n_j} {\mbox{ s.t. }}
\rho_1<\dots<\rho_{n_j}\}.$$
Then, since~$
p^{(k)}(t_k)\in\partial {\mathcal{D}}_k$,
if~$t_k$ is finite, we find~$n_{k+1}\le n_k-1$
such that
the set of points~$\{p^{(k)}_1(t_k),\dots,p^{(k)}_{n_k}(t_k) \}$
coincides
with
a set of~$n_{k+1}$ distinct points, that we denote by~$\{p^{(k+1)}_1(t_k),\dots,p^{(k+1)}_{n_k}(t_k) \}$, 
with the convention that
$$ p^{(k+1)}_1(t_k)<\dots<p^{(k+1)}_{n_k}(t_k).$$
For any~$i\in\{1,\dots,n_{k+1}\}$, we set\footnote{It is useful
to observe that, in light of~\eqref{numero},
$$ \sum_{i\in \{1,\dots,n_{k+1}\} } \bar w_{i,k+1}=
\sum_{i\in \{1,\dots,n_{k}\} } \bar w_{i,k},$$ \label{OGG}
which says that the total number of little penguins remains always the same
(more precisely, the sum of all the little penguins in all groups
is constant in time).}
\begin{equation}\label{numero} \bar w_{i,k+1} := \sum_{{j\in\{1,\dots,n_k\}}\atop{
p^{(k)}_j(t_k)=p^{(k+1)}_i(t_k)
}} \bar w_{j,k}.\end{equation}
Let also~$ \bar\wgotica_{i,k+1}=\wgotica(\bar w_{i,k+1})$.
Then, for any~$i\in\{1,\dots,n_{k+1}\}$ and any~$\rho=
(\rho_1,\dots,\rho_{n_{k+1}})\in\R^{n_{k+1}}$, we set
\begin{equation*}  \pan_{i,k+1}(\rho)
:=
\max\Big\{ \bar\wgotica_{i,k+1},\;
\max_{ {j\in\{1,\dots,n_{k+1}\}}\atop{j\ne i}}
\varphi\big(|\rho_i-\rho_j|\big)
\Big\}. \end{equation*} 
For any~$i\in\{1,\dots,n_{k+1}\}$ we also
define
$$\bar\mu_{i,k+1}:=\mu(\bar w_{i,k+1}),$$ where~$\mu$ is the function defined in~\eqref{MU}
and, for any~$\rho\in\R^{n_{k+1}}$,
\begin{equation*} 
\bar m_{i,k+1}(\rho):=
\sum_{j\in\{1,\dots,n_{k+1}\}} {\rm sign}\,(\rho_j-\rho_i)\;\bar w_{j,k+1}\;
\sgotica(|\rho_i-\rho_j|)
.\end{equation*}
We notice that~$
\bar m_{i,k+1}(\rho)$ is bounded and Lipschitz for any~$\rho\in
{\mathcal{D}}_{k+1}:=\{\rho=(\rho_1,\dots,\rho_{n_{k+1}})\in\R^{n_{k+1}} {\mbox{ s.t. }}
\rho_1<\dots<\rho_{n_{k+1}}\}$.

We also define
\begin{equation*}
\vel_{i,k+1}(\rho):=(1-\bar\mu_{i,k+1})
\,\bar m_{i,k+1}(\rho)
+v\bar\mu_{i,k+1}\end{equation*}
and
$$ G_{i,k+1}(\rho,t):=\pan_{i,k+1}(\rho)\,\big(\eps+\vel_{i,k+1}(\rho)\big)+
f(\rho_i,t).$$
In this way,
we have that~$G_{i,k+1}$
is bounded and Lipschitz in~${\mathcal{D}}_{k+1}\times[0,+\infty)$
and so we find the next solution~$p^{(k+1)}(t)=(p^{(k+1)}_1(t),\dots,p^{(k+1)}_{n_{k+1}}(t))\in {\mathcal{D}}_{k+1}$
in the interval~$(t_k,t_{k+1})$, with~$p^{(k+1)}(t_{k+1})\in\partial
{\mathcal{D}}_{k+1}$, by solving the ordinary differential equation
$$ \dot p^{(k+1)}_i(t)= G_{i,k+1}\big(p^{(k+1)}(t),t\big). $$
This completes the iteration argument
and provides the desired notion of solution for equation~\eqref{EQ}.

\section{Examples of safe return home}\label{HOME}

Here, we provide some sufficient conditions for
the penguins to reach their home, located at the point~$H$
(let us mention that,
in the parade that we saw live,
one little penguin remained stuck
into panic and did not manage to return home -- so,
giving a mathematical treatment of the case in which
the strategy of the penguins turns out to be successful
somehow reassured
us on the fate of the species).

To give a mathematical framework of the notion of homecoming,
we introduce the function
$$ [0,+\infty)\ni t\mapsto {\mathcal{N}}(t):=
\sum_{{j\in\{1,\dots,n(t)\}}\atop{p_j(t)=H}} w_j(t).$$
In the setting of footnote~\ref{STAY},
the function~${\mathcal{N}}(t)$ represents the number of penguins
that have safely returned home at
time~$t$.

For counting reasons, we also point out that
the total number of penguins is constant and given by
$$ {\mathcal{M}}:=\sum_{ {j\in\{1,\dots,n(0)\}} } w_j(0)
=\sum_{ {j\in\{1,\dots,n(t)\}} } w_j(t),$$
for any~$t\ge0$ (recall footnote~\ref{OGG}).

The first result that we present says that if at some time
the group of penguins that stay
further behind
gathers into a group of at least two elements, then
all the penguins will manage to eventually return home.
The mathematical setting goes as follows:

\begin{theorem}\label{THM:1}
Let~$t_o\ge0$ and assume that
\begin{equation}\label{jk:11}
\eps+\inf_{(r,t)\in \R\times[t_o,+\infty)}
f(r,t)\ge \iota 
\end{equation}
for some~$\iota>0$, and
\begin{equation}\label{jk:12}
w_1(t_o)\ge2.
\end{equation}
Then, there exists~$T\in \left[ t_o, \,t_o+\frac{H-p_1(t_o)}{\iota}\right]$
such that
$$ {\mathcal{N}}(T)={\mathcal{M}}.$$
\end{theorem}

\begin{proof} We observe that~$w_1(t)$ is nondecreasing in~$t$,
thanks to~\eqref{numero}, and therefore~\eqref{jk:12}
implies that~$w_1(t)\ge2$ for any~$t
\ge t_o$. Consequently, from~\eqref{OMEGA},
we obtain that~$ \wgotica(w_1(t))=1$
for any~$t
\ge t_o$. This and~\eqref{PAN} give that~$
\pan_1\big(\rho,w(t);t\big) =1$ for any~$t\ge t_o$ and any~$\rho\in\R^{n(t)}$.
Accordingly, the equation of motions in~\eqref{EQ} 
gives that, for any~$t\ge t_o$,
\begin{equation*}
\dot{p}_1(t) =\eps+\vel_1\big(
p(t),w(t);t\big)+f\big(p_1(t),t\big)\ge \eps+f\big(p_1(t),t\big)\ge
\iota,
\end{equation*}
thanks to~\eqref{jk:11}. That is, for any~$j\in\{1,\dots,n(t)\}$,
$$ p_j(t)\ge p_1(t)\ge \min\{ H, \; p_1(t_o)+\iota \,(t-t_o) \},$$
which gives the desired result.
\end{proof}

A simple variation of Theorem~\ref{THM:1}
says that if, at some time, a group of little penguins
reaches a sufficiently large size, then all the penguins in this group
(as well as the ones ahead) safely reach their home.
The precise statement (whose proof is similar to the one of
Theorem~\ref{THM:1}, up to technical modifications, and is therefore
omitted) goes as follows:

\begin{theorem}
Let~$t_o\ge0$ and assume that
\begin{equation*}
\eps+v+\inf_{(r,t)\in \R\times[t_o,+\infty)}
f(r,t)\ge \iota 
\end{equation*}
for some~$\iota>0$, and
\begin{equation*}
w_{j_o}(t_o)\ge\kappa,
\end{equation*}
for some~$j_o\in\{1,\dots,n(t_o)\}$.

Then, there exists~$T\in \left[ t_o, \,t_o+\frac{H-p_{j_o}(t_o)}{\iota}\right]$
such that
$$ {\mathcal{N}}(T)\ge 
\sum_{ {j\in\{j_o,\dots,n(t_o)\}} } w_j(t_o)
.$$
\end{theorem}

\section{Pictures, videos and numerics}\label{VIDEO}

In this section, we present some simple numerical experiments
to facilitate the intuition at the base of the model presented in~\eqref{EQ}.
These simulations may actually be easily compared with the ``real life'' experience
and indeed they show some of the typical treats of the little penguins parade,
such as the
oscillations and sudden change of direction, the gathering of the penguins into clusters and the possibility that some elements of the herd remain isolated and panic,
either\footnote{The possibility that a penguin remains isolated
also in the sea may actually occur in the real-world experience, as demonstrated
by the last penguin in the herd on the video (courtesy of Phillip Island Nature Parks)
available online at the webpage\\
{\tt 
https://www.ma.utexas.edu/users/enrico/penguins/Penguins2.MOV} }
on the land or in the sea.

In our simulations, for the sake of simplicity, we considered 20 penguins
returning to their burrows from the shore -- some of the penguins may start
their trip from the sea (that occupies the region below level~$0$
in the simulations) in which waves and currents may affect the
movements of the animals. The pictures that we produce have the time variable
on the horizontal axis and the space variable on the vertical axis
(with the burrow of the penguins community set at level~$4$
for definiteness). The pictures are,
somehow, self-explanatory. For instance, in Figure~\ref{MP5},
we present a case in which, fortunately, all the little penguins
manage to safely return home,
after having gathered into groups:
as a matter of fact, in the first of these pictures
all the penguins safely reach home together at the same time
(after having rescued the first penguin, who stayed still for a long
period due to isolation and panic); on the other hand,
the second of these pictures
shows
that a first group of penguins, which was originated
by the animals that were on the land at the initial time,
reaches home
slightly before the second group of penguins, which was originated
by the animals that were in the sea
at the initial time (notice also that
the motion of the penguins in the sea
appears to be affected by waves and currents).

\begin{figure}
    \centering
    \includegraphics[width=10.8cm]{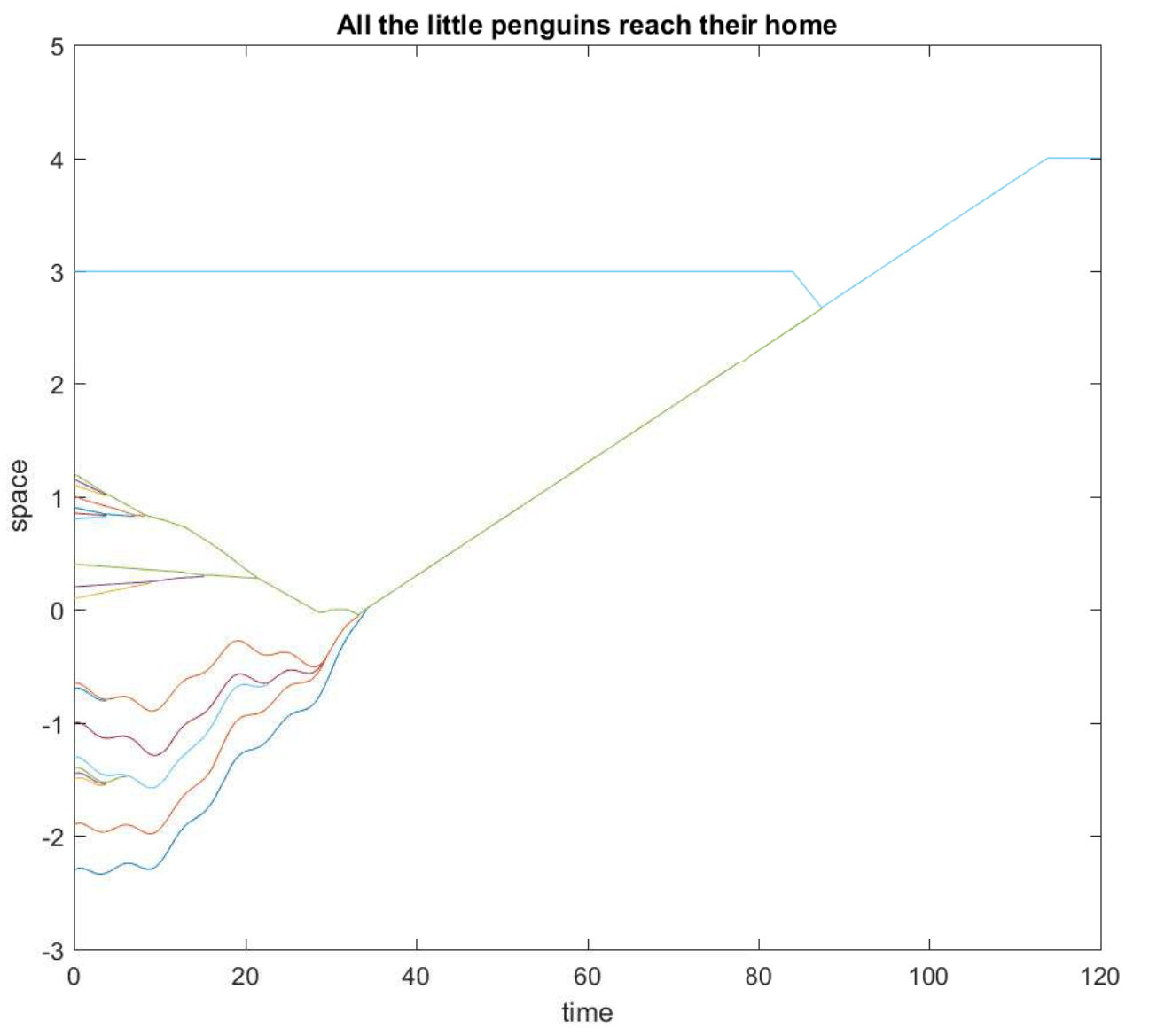}
    \includegraphics[width=10.8cm]{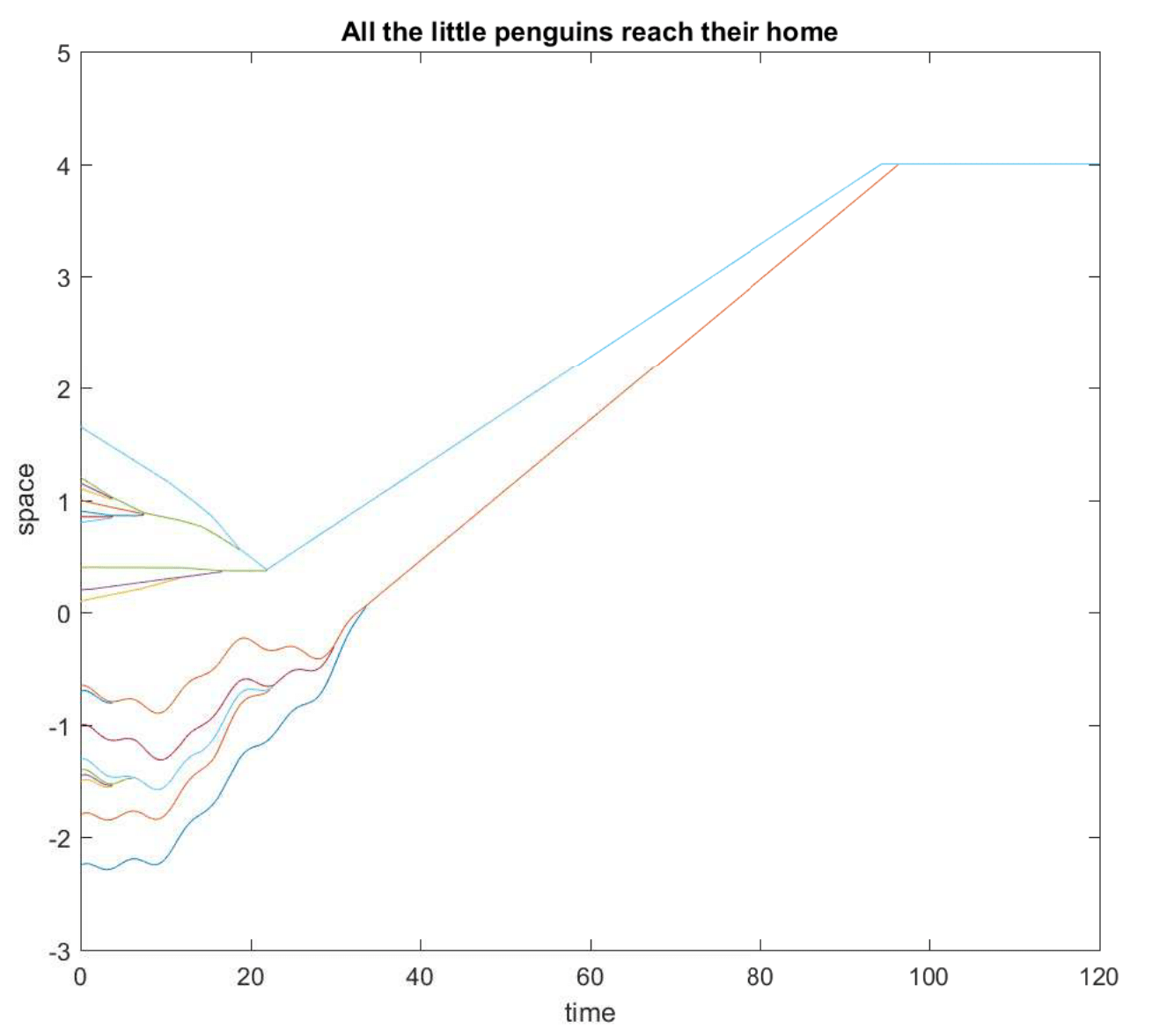}
    \caption{All the little penguins safely return home.}
    \label{MP5}
\end{figure}

We also observe a different scenario depicted in Figure~\ref{MK1}
(with two different functions to represent the currents in the sea):
in this situation, a big group of $18$ penguins
gathers together (collecting also penguins who were initially in the water)
and safely returns home. Two penguins remain isolated in the water,
and they keep slowly moving towards their final destination (that they
eventually reach after a longer time).

\begin{figure}
    \centering
    \includegraphics[width=10.8cm]{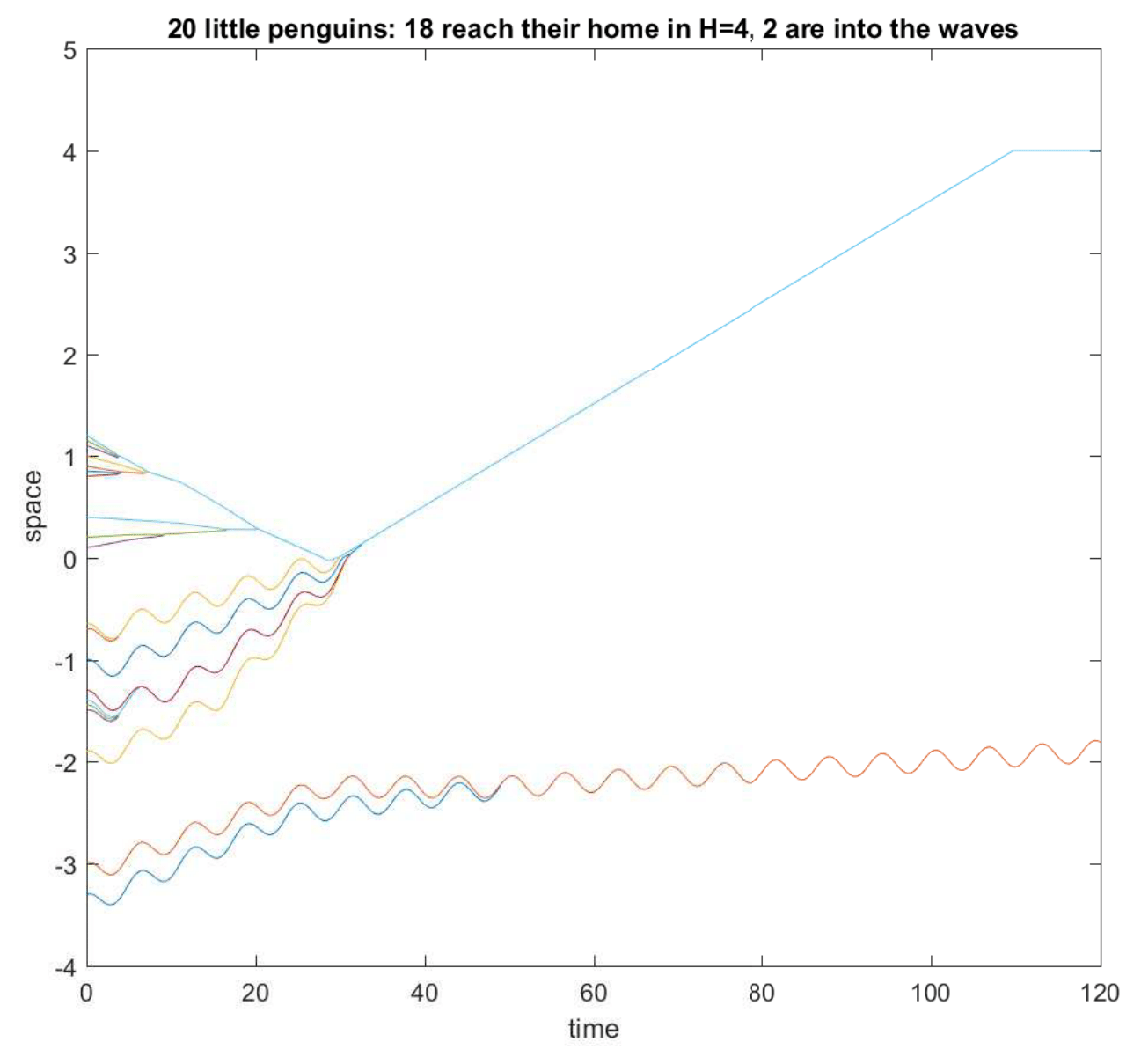}
        \includegraphics[width=10.8cm]{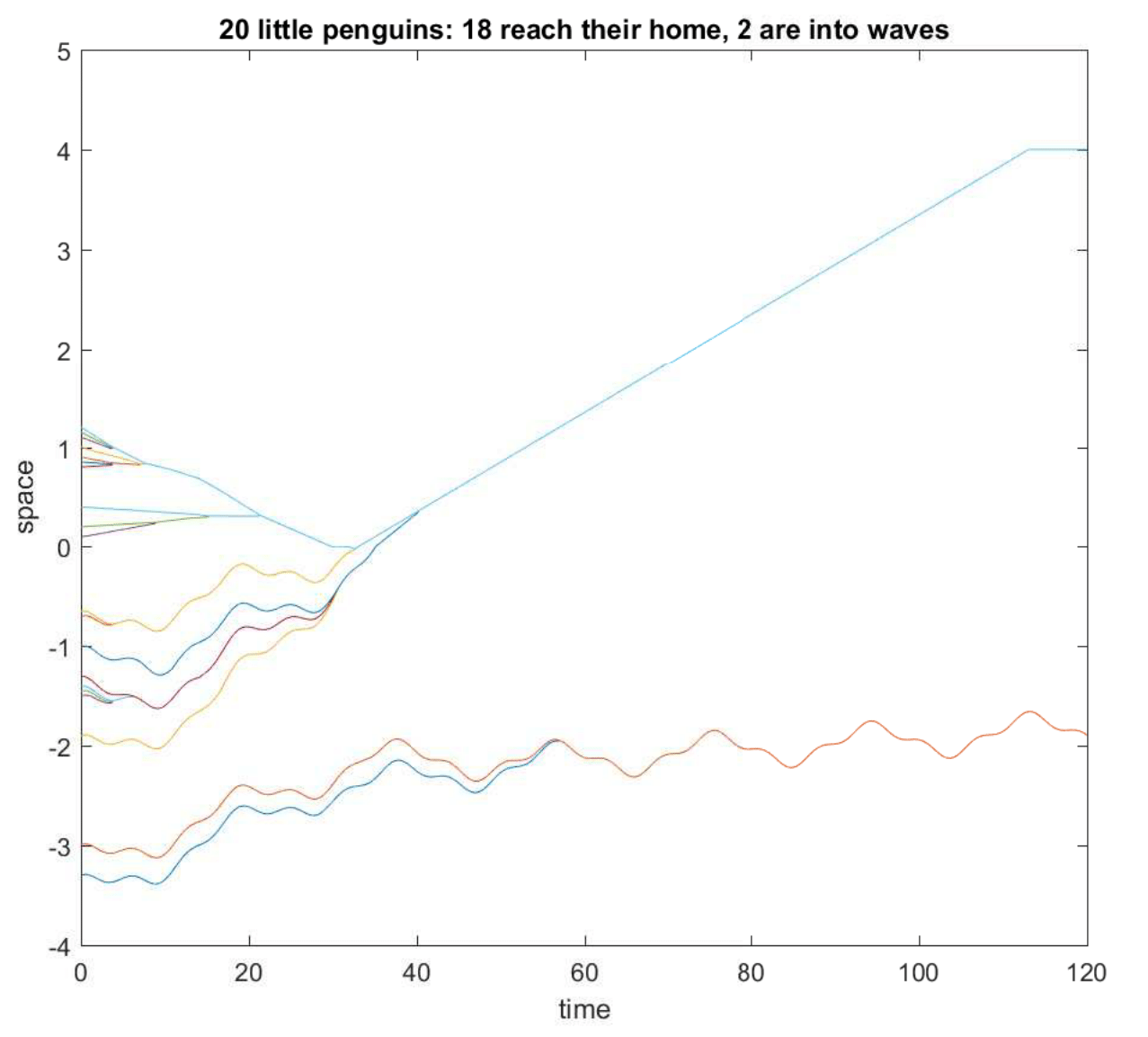}
    \caption{Two penguins are still in the water after a long time.}
    \label{MK1}
\end{figure}

Similarly, in Figure~\ref{MK12}, almost all the penguins gather into a single group
and reach home, while two penguins get together in the sea, they come to
the shore and slowly waddle towards their final destination,
and one single penguin remains isolated and panics in the water, moved by the currents.

\begin{figure}
    \centering
    \includegraphics[width=10.8cm]{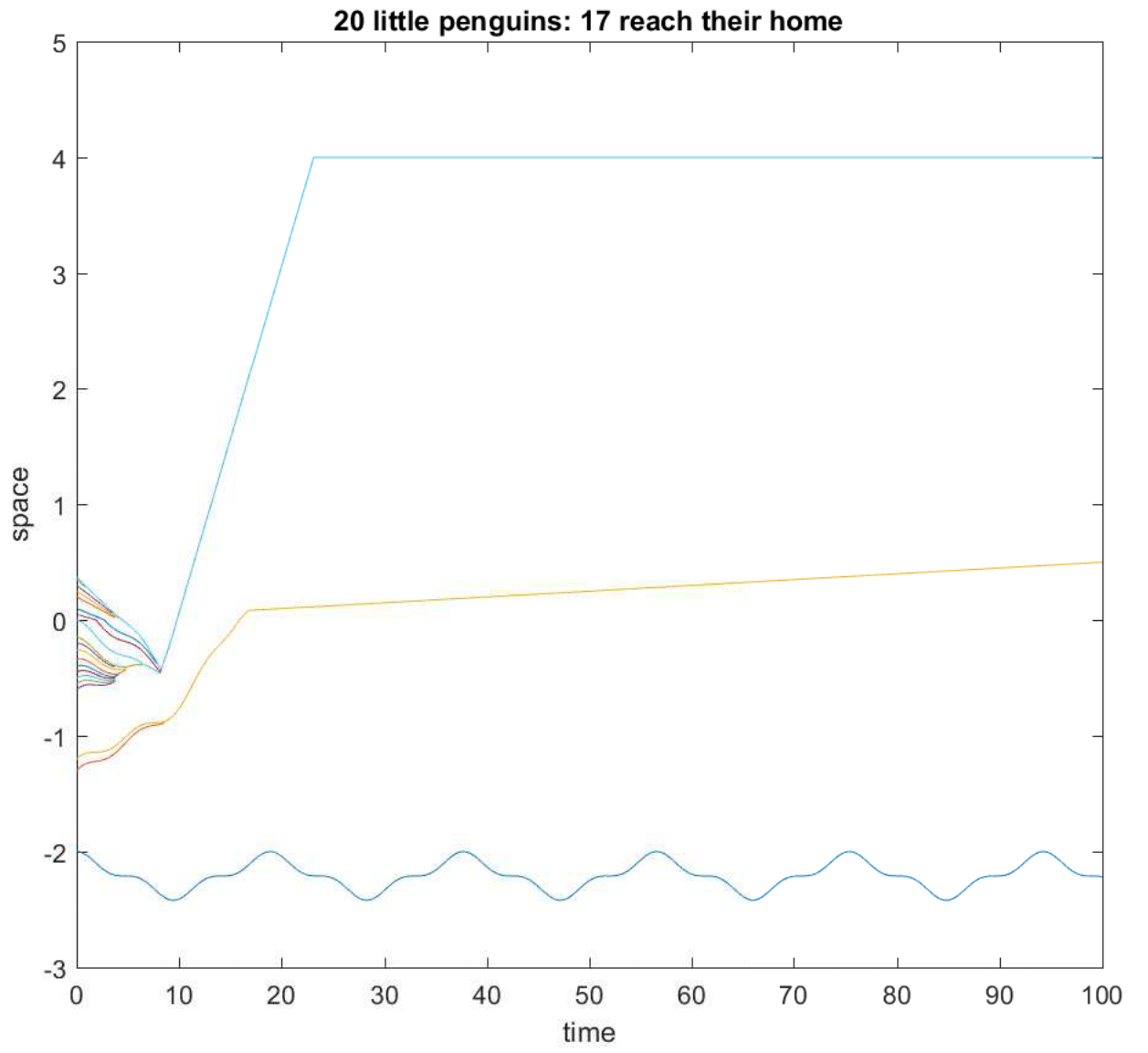}
    \caption{One penguin is stuck in the water.}
    \label{MK12}
\end{figure}

The situation in Figure~\ref{MK4} is slightly different, since
the last penguin at the beginning moves towards the others, but (s)he
does not manage to join the forming group by the time the other penguins
decide to move consistently towards their burrows -- so, unfortunately
this last penguin, in spite of the initial effort, finally remains stuck in the water.
  
\begin{figure}
    \centering
    \includegraphics[width=10.8cm]{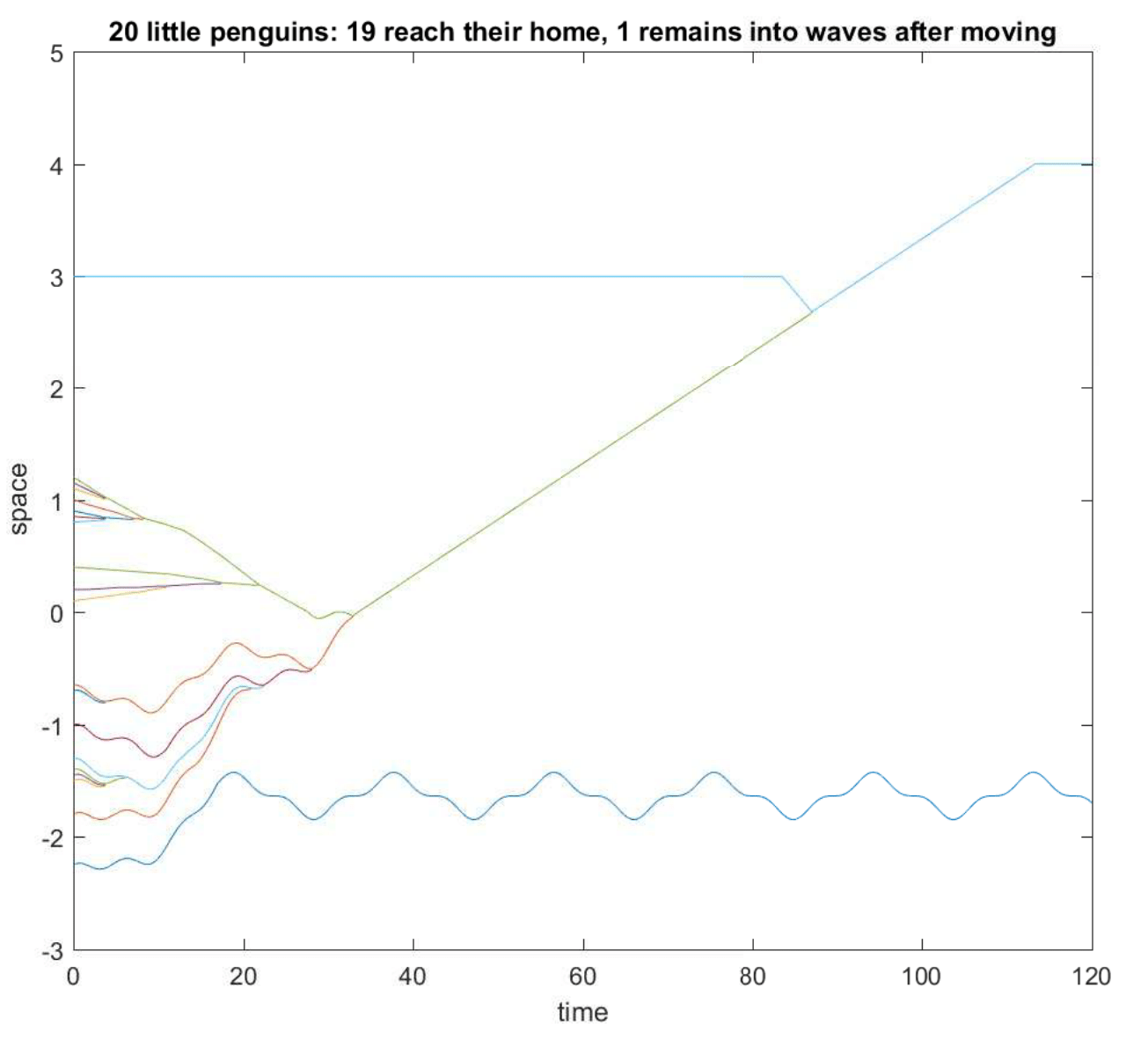}
    \caption{One penguin moves towards the others but remains stuck in the water.}
    \label{MK4}
\end{figure}

In Figure~\ref{MK6}, all the penguins reach their burrows, with the exception of the
last two ones: at the time we end the simulation, one penguin
is stuck on the shore, due to panic, and another one is very slowly approaching
the beach, but (s)he is still in the water (small modifications of the initial
conditions and of the wave function may lead to different future outcomes, namely
either the last penguin is able to reach the shore and happily meet the other mate to waddle
together home, or the strong current may prevent the last penguin to reach the beach,
in which case also the penguin in front would remain stuck).

\begin{figure}
    \centering
    \includegraphics[width=10.8cm]{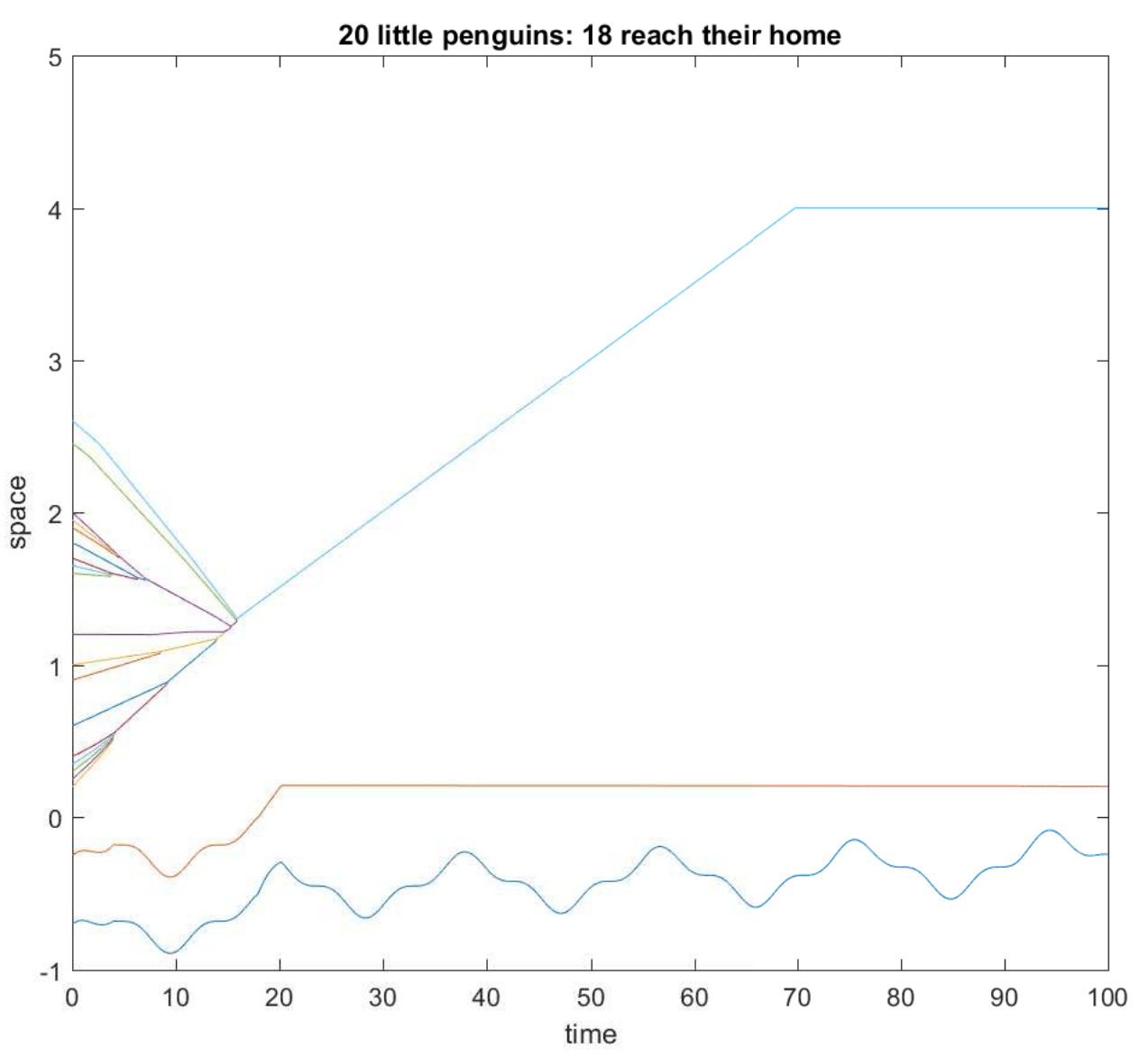}
    \caption{One penguin freezes on the shore, another stays in the water.}
    \label{MK6}
\end{figure}

With simple modifications of the function $f$, one can also consider the
case in which the waves of the sea change with time and their influence
may become more (or less) relevant for the swimming of the little penguins:
as an example of this feature, see Figure \ref{ON}.

\begin{figure}
    \centering
    \includegraphics[width=8.8cm]{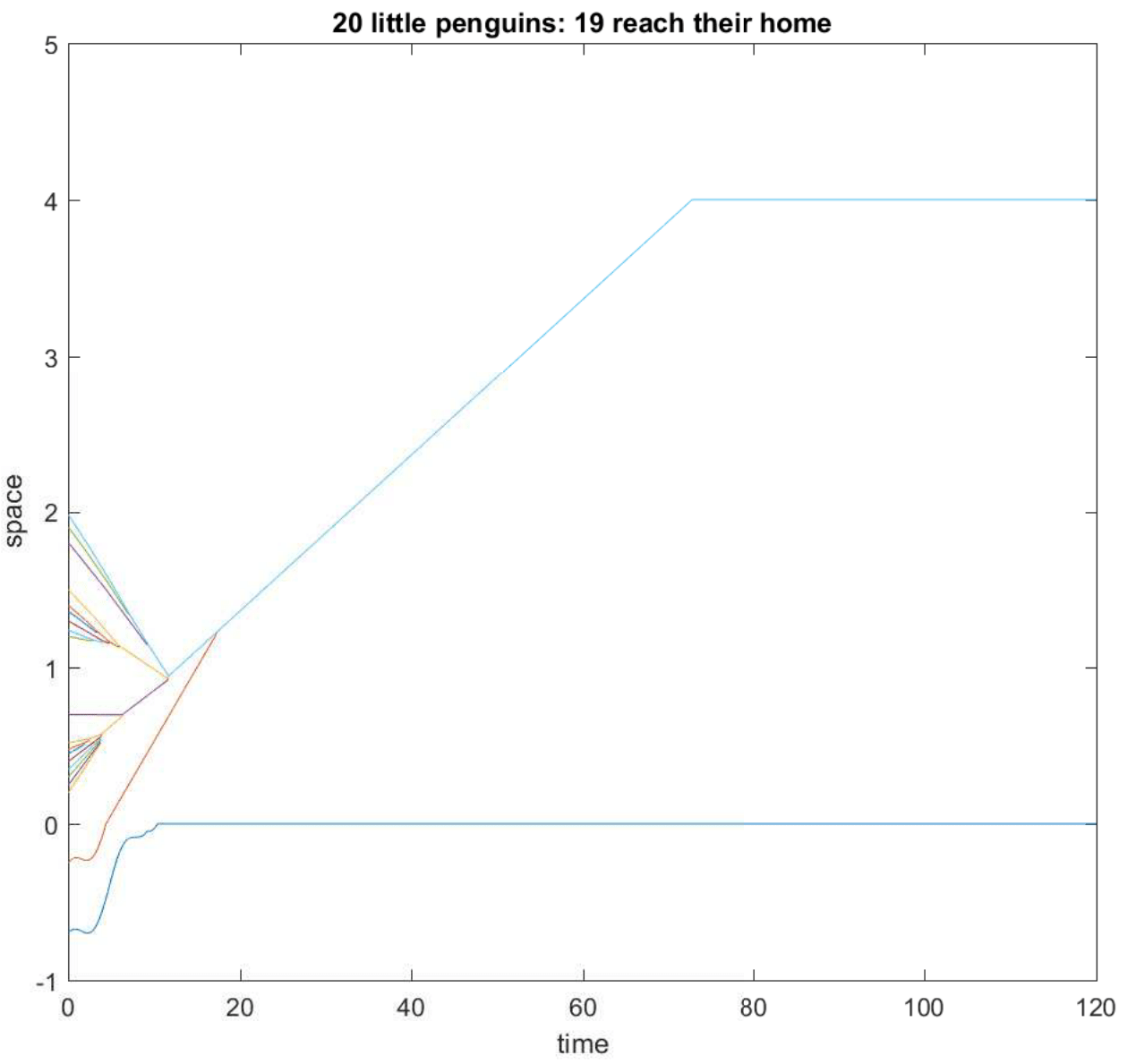}
    \includegraphics[width=8.8cm]{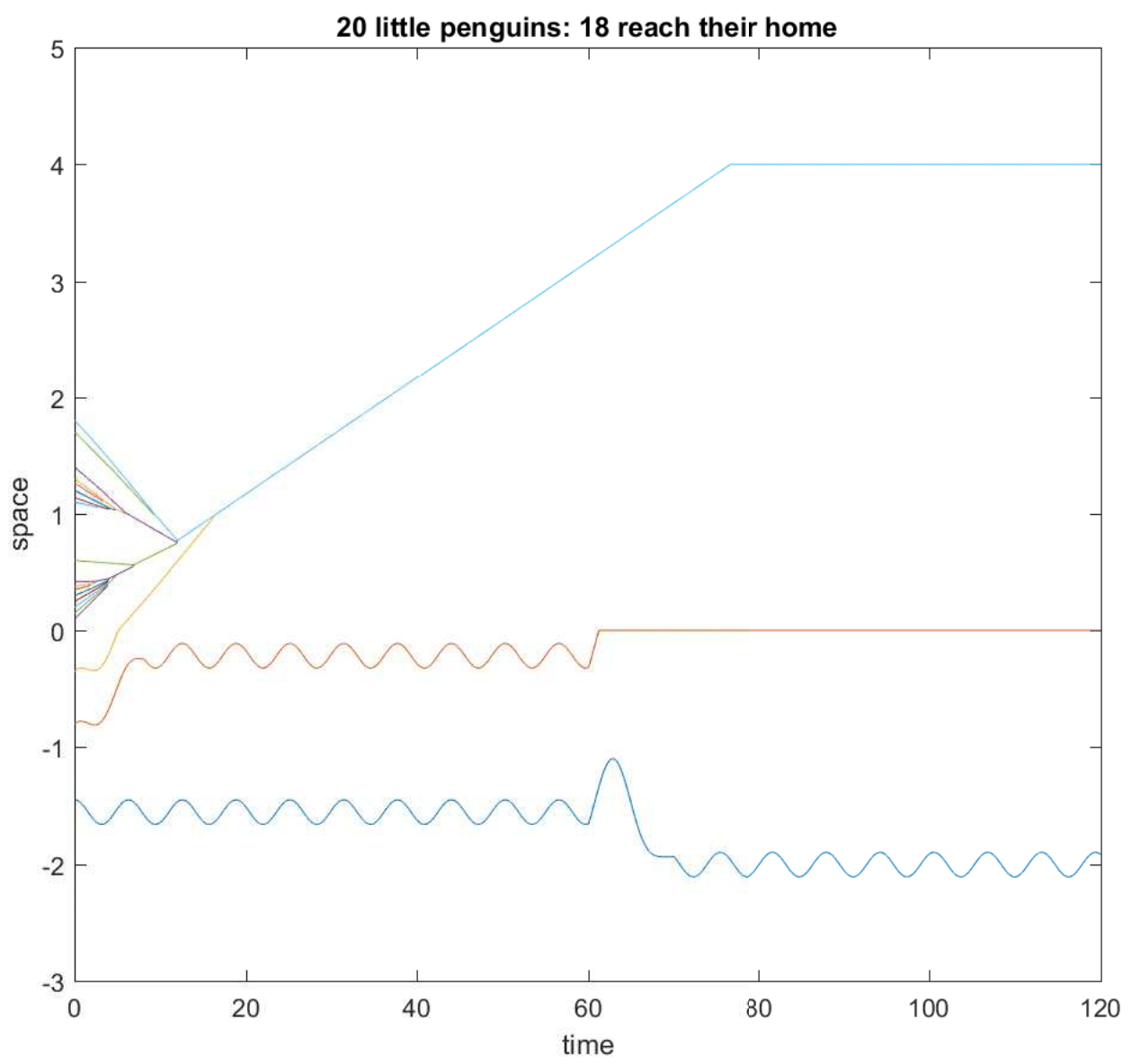}
    \includegraphics[width=8.8cm]{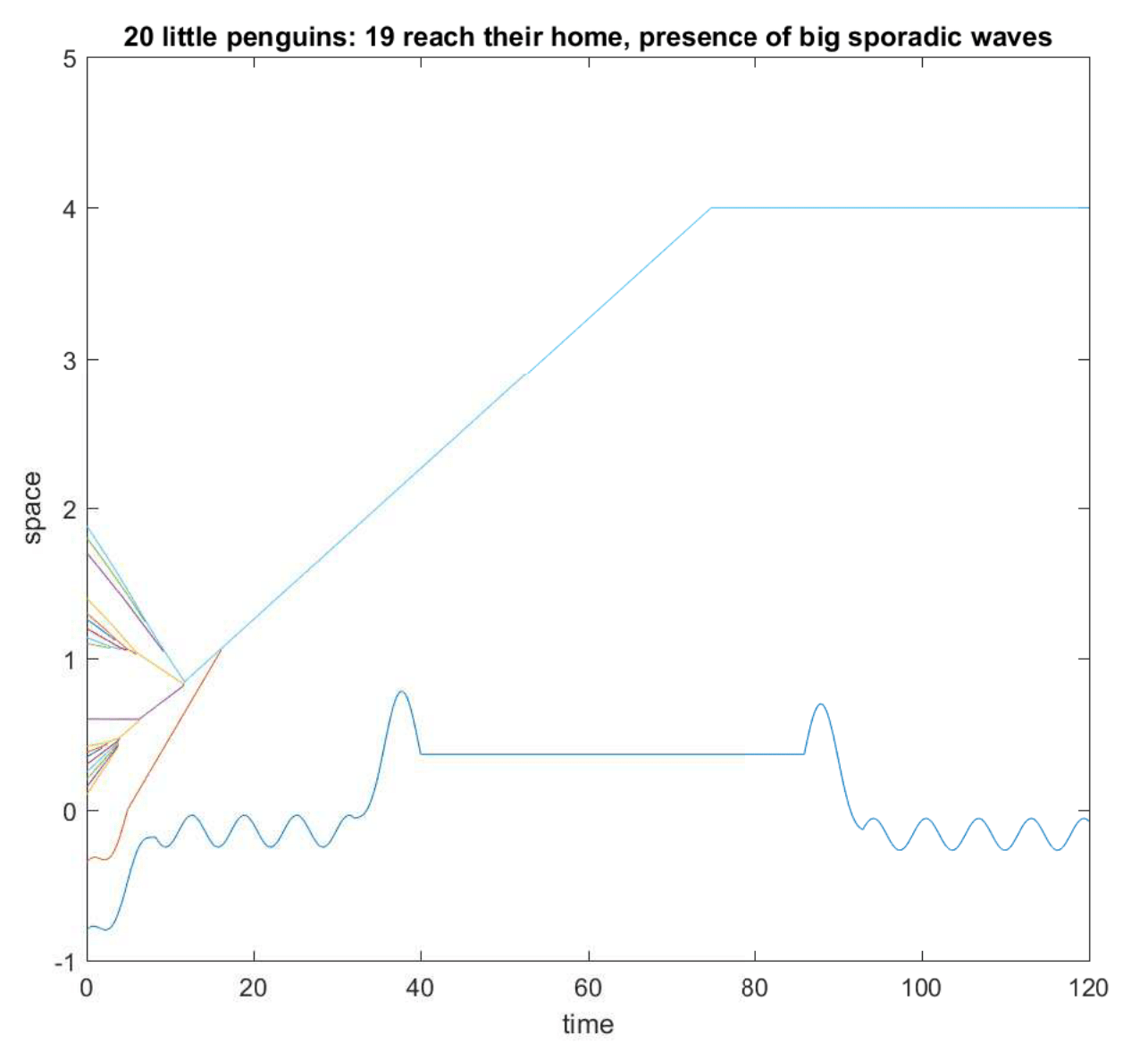}
    \caption{Effect of the waves on the movement of the penguins in the sea.}
    \label{ON}
\end{figure}

Finally, we recall that, in the setting of Section~\ref{INTRO},
once a group of little penguins is created, then it moves consistently
altogether. This is of course a simplifying assumption,
and it might happen in reality that one or a few penguins
leave a large group after its formation
-- perhaps because one penguin is 
slower than the other penguins
of the group, perhaps because (s)he gets distracted
by other events on the beach, or simply because
(s)he feels too exposed being at the side of the group
and may prefer to form a new group in which (s)he
finds a more central and protected position. Though we plan to describe
this case in detail in a forthcoming project (also possibly 
in light of morphological and social considerations
and taking into account a possible randomness
in the system), we stress that natural modifications
can be implemented inside the setting of Section~\ref{INTRO}
to take into account also this feature. For simple and concrete
examples, see Figure~\ref{LEA},
in which several cases are considered
(e.g., one of the little penguins leaving the group gets stuck,
or goes back into the water, or meets another little penguin,
and so on).

\begin{figure}
    \centering
    \includegraphics[width=8.8cm]{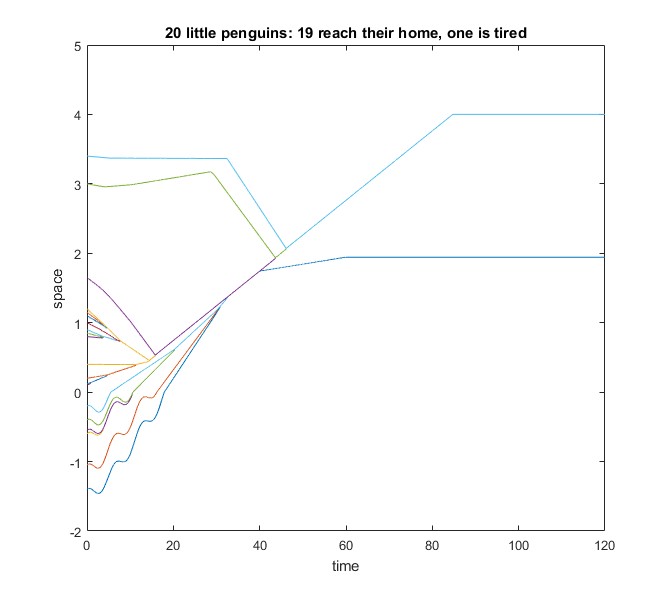}
    \includegraphics[width=8.8cm]{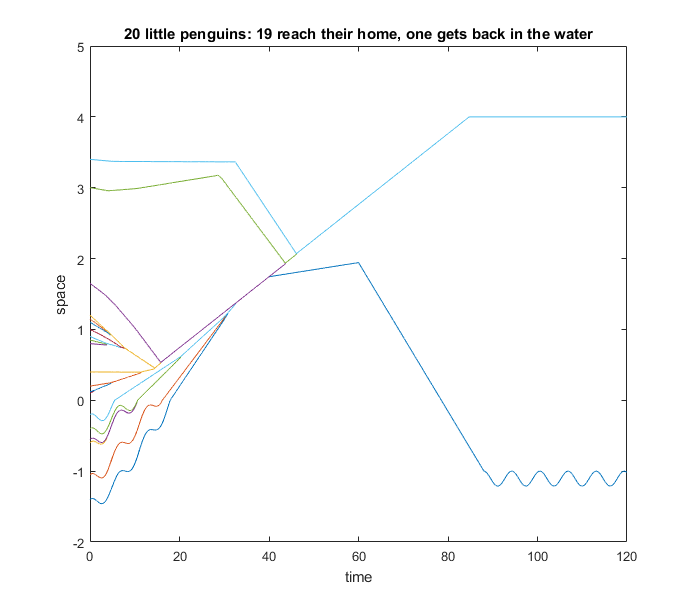}
    \includegraphics[width=8.8cm]{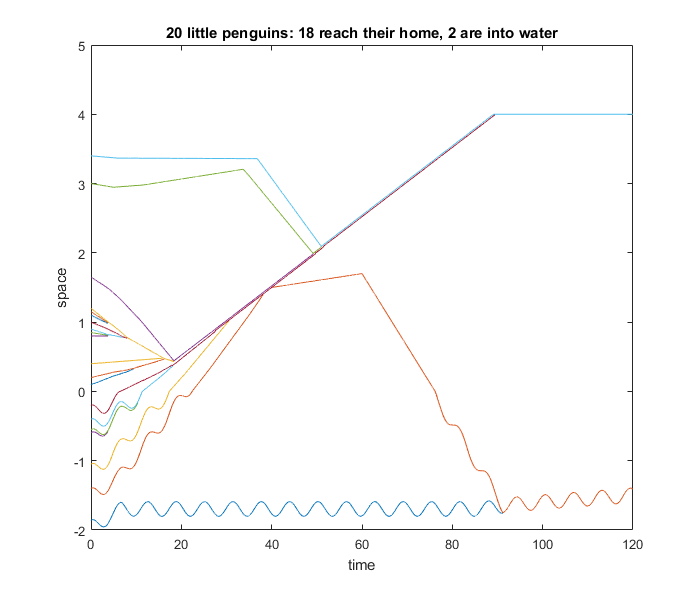}
    \caption{A modification: one little penguin may leave the group.}
    \label{LEA}
\end{figure}
\medskip

The situation in which one little penguin seems to think about
leaving an already formed group can be observed in the video (courtesy of Phillip Island Nature Parks)\\
{\tt
https://www.ma.utexas.edu/users/enrico/penguins/Penguins2.MOV} \\
(see in particular the behavior of the
second penguin from the bottom, i.e.
the last penguin of 
the already formed large cluster).
\medskip

We point out that all these pictures have been easily
obtained by short programs
in MathLab.
As an example,
we posted one of the source codes 
of these programs
on the webpage\\
{\tt
https://www.ma.utexas.edu/users/enrico/penguins/cononda.txt} \\
and all the others are available upon request
(the simplicity of these programs shows that the model in~\eqref{EQ}
is indeed very simple to implement numerically, still producing 
sufficiently ``realistic'' results
in terms of cluster formation and cruising speed of the groups).

Also, these pictures can be easily translated into animations. Simple videos
that we have
obtained by these numerics are available from the webpage\\
{\tt
https://www.ma.utexas.edu/users/enrico/penguins/VID/ }

\section*{Acknowledgements}
This work has been
supported by the ERC grant $\epsilon$ ({\it Elliptic Pde's and
Symmetry of Interfaces and Layers for Odd Nonlinearities})
and the PRIN grant
201274FYK7 (Critical Point Theory
and Perturbative Methods for Nonlinear Differential Equations).
Part of this paper has been written on the occasion of 
a very pleasant visit
of the second and third authors to the University of Melbourne.
We thank Claudia Bucur and Carina Geldhauser for their interesting comments
on a preliminary version of this manuscript.
We would also like to thank Andre Chiaradia for a very instructive
conversation and for sharing the bibliographic information
and some videos of the little penguins parade. Videos credit
to Phillip Island Nature Parks.

\end{document}